\documentclass[10pt,oneside]{article}

\usepackage{amscd,amssymb}
\usepackage{dsfont,amsthm,amssymb,amsmath,amsfonts}
\usepackage{graphicx}
\graphicspath{ {images/} }
\usepackage{minitoc}
\usepackage{authblk}
\usepackage{tikz}
\usepackage{mathptmx}
\usepackage{float}
\usepackage{ragged2e}
\usepackage{lipsum}
\usepackage[utf8]{inputenc}
\usepackage[english]{babel}

\newtheorem{theorem}{Theorem}[section]
\newtheorem{lemma}[theorem]{Lemma}
\newtheorem{corollary}[theorem]{Corollary}

\newtheorem{remark}[theorem]{Remark}
\newtheorem{example}[theorem]{Example}
\newtheorem{definition}[theorem]{Definition}

\newcommand{\gr}{\textrm{\hbox{\textup{gr}}}}
\newcommand{\dm}{\textrm{\hbox{\textup{diam}}}}
\newcommand{\nl}{\textrm{\hbox{\textup{Nil}}}}
\newcommand{\ann}{\hbox{\textrm{\hbox{\textup{ann}}}}}
\providecommand{\kw}[1]{\par
	\vspace*{8pt}
	{\footnotesize{\leftskip18pt\rightskip\leftskip
	\noindent #1\par}}\par}
\providecommand{\ccode}[1]{\par
	\vspace*{8pt}
	{\footnotesize{\leftskip18pt\rightskip\leftskip
	\noindent #1\par}}\par}
\providecommand{\abstrct}[1]{
	\vspace*{10pt}
	{\footnotesize{\leftskip18pt\rightskip\leftskip
	\noindent #1\par}}\par}

\date{ }

\begin{document}


\title{{\textbf{Compressed Intersection Annihilator Graph}}}
\author{{\textbf{Mayssa Soliman$^1$,} \textbf{Nefertiti Megahed$^2$}}}
\affil{\textit{$^1$ Department of Mathematics, Faculty of Science, Cairo University,}\\ \textit{Giza, Egypt} \\ \textit{mayssaabdelhamed@gstd.sci.cu.edu.eg}}
\affil{\textit{$^2$ Department of Mathematics, Faculty of Science, Cairo University,}\\ \textit{Giza, Egypt} \\ \textit{nefertiti@sci.cu.edu.eg}}
\maketitle
\abstrct{
Let $R$ be a commutative ring with a non-zero identity. In this paper, we define a new graph, the compressed intersection annihilator graph, denoted by $IA(R)$, and investigate some of its theoretical properties and its relation with the structure of the ring. It is a generalization of the torsion graph $\Gamma_{R}(R)$. We study classes of rings for which the equivalence between the set of zero-divisors of $R$ being an ideal and the completeness of $IA(R)$ holds. We also study the relation between $\Gamma_{R}(R)$ and $IA(R)$. In addition, we show that if the compressed intersection annihilator graph of a ring $R$ is finite, then there exist a subring $S$ of $R$ such that $IA(S)\cong IA(R)$. Also, we show that the compressed intersection annihilator graph will never be a complete bipartite graph. Besides, we show that the graph $IA(R)$ with at least three vertices is connected and its diameter is less than or equal to three. Finally, we determine the properties of the graph in the cases when $R$ is the ring of integers modulo $n$, the direct product of integral domains, the direct product of Artinine local rings and the direct product of two rings such that one of them is not an integral domain.}

\kw{\textit{Keywords}: Annihilator; annihilator graph; torsion graph; compressed annihilator graph.}
\ccode{2010 Mathematics Subject Classification: 05C25, 05C99, 13M99, 13A99, 13A15}


\begin{Section}{Introduction}
\noindent
The study of zero-divisors plays an important role in the ring theory, for example to find solutions to equations, however the set of zero-divisors lacks algebraic structure. The set of zero-divisors of a ring $R$, denoted by $Z(R)$, is always closed under multiplication but it is not always closed under addition. In addition the multiplication operation is not transitive. In order to try to understand the relation between $Z(R)$ elements through graph theory a new approach was developed recently. We assume throughout this paper that $R$ is a commutative ring with non-zero identity.

The interplay between ring theory and graph theory started from the work of I.~Beck in 1988, \cite{Beck}, when he defined the zero-divisor graph $\Gamma(R)$ as the undirected simple graph with vertices represented by all elements of $R$ and two distinct vertices are adjacent if their product is zero. In 1999, D.~F.~Anderson and P.~S.~Livingston in \cite{LivingstonI} made a modification to Beck's graph by considering the set of vertices to be the set of non-zero zero-divisors of the ring $R$ and the adjacency kept as before. They were interested in determining some important properties of the graph and their relation to $R$. Inspired by ideas of S.~B.~Mulay in \cite{Mulay}, S.~Spiroff and C.~Wickham in 2011, \cite{compressed}, modified again the set of vertices by introducing the concept of the compressed zero-divisor graph denoted by $\Gamma_{E}(R)$. The set of vertices was constructed from equivalence classes of zero-divisors determined by the following equivalence relation $\sim$ on $R$: $x\sim y$ if and only if $\ann_{R}(x)=\ann_{R}(y)$, for any $x,\,y \in R$, where $\ann_{R}(x)=\{r\in R\,\mid\, rx=0\}$. E.~Lewis in \cite{Elizebth} proved that $\sim$ is a multiplicative congruence relation on $R$. By the definition of the relation, we have ${[0]_{\sim}=\{x\in R\mid \ann_{R}(x)=R\}=\{0\}}$, $[1]_{\sim}=\{{x\in R}\mid \ann_{R}(x)=\{0\}\}=R\setminus Z(R)$ and thus we have $[x]_{\sim}\subseteq {Z^{*}(R)=Z(R)\setminus\{0\}}$ for each $x\in R\setminus ([0]_{\sim}\cup[1]_{\sim})$. The union here means union of sets not of classes. She also, proved that the set ${R/{\sim}}=\{[x]_{\sim}\mid x\in R\}$ of all congruence classes with the well-defined multiplication, given by $[x]_{\sim}[y]_{\sim}=[xy]_{\sim}$ for all $x, y\in R$, is a commutative monoid with identity $[1]_{\sim}$ and zero $[0]_{\sim}$. S.~Spiroff and C.~Wickham defined $\Gamma_{E}(R)$ as the undirected simple graph with the set of vertices $Z(R/\sim)^{*}=\{[x]_{\sim}\,\mid\,x\in R\setminus ([0]_{\sim}\cup[1]_{\sim})\}$ and two distinct vertices $[x]_{\sim}$ and $[y]_{\sim}$ are adjacent if and only if $xy=0$. There are many different ways of generalizations of the zero-divisor graph, see for example \cite{Non-comm, ZeroModule, ideal-based}.
\par
This is not the only way to associate a graph to a ring and vise versa. One of those important ways introduced by D.~F.~Anderson and A.~Badawi in 2008, \cite{Total}, is the total graph of a commutative ring denoted by $T(\Gamma(R))$. It is defined as the undirected simple graph with $R$ as the set of vertices and two distinct vertices are adjacent if and only if their sum is a zero-divisor. They characterized the properties of the graph when $Z(R)$ is either an ideal or not. In 2013, D.~F.~Anderson and A.~Badawi in \cite{GTotalGph} generalized the total graph over a commutative ring $R$ with respect to a multiplicative prime subset $H$\footnote{{A non-empty proper subset $H$ of $R$ is said to be a \textit{multiplicative-prime subset} of $R$ if it satisfies the following two conditions:} {\begin{enumerate}
                  \item {$ab\in H$ for every $a\in H$ and $b\in R$},\,
                  \item {if $ab\in H$ for $a,b\in R$, then either $a\in H$ or $b\in H$.}
                \end{enumerate}}
As an example of a multiplicative-prime subset, we take $H=Z(R)$.} of $R$ and there are ways of generalizations of the total graph. For surveys on this topic see also \cite{SurveyAB, SurveyKN}.
\par
Another way to associate a graph with a ring was given by A.~Badawi in 2014, \cite{Annih}, where he introduced the annihilator graph denoted by $AG(R)$. The set of vertices is the same as the set of vertices of zero-divisor graph $Z^*(R)$, but two distinct vertices $x$ and $y$ are adjacent if and only if $\ann_{R}(xy)\neq \ann_{R}(x)\cup \ann_{R}(y)$. It is clear that $\ann_{R}(x)\cup \ann_{R}(y)\subset \ann_{R}(xy)$ but the equality does not hold in general. In 2018, Sh.~Payrovi and S.~Babaei in \cite{Gen.ann} generalized $AG(R)$ to be the compressed annihilator graph $AG_{E}(R)$. The set of vertices is the set of equivalent classes of zero-divisors of $R$, $Z(R/\sim)^*$ and two distinct vertices $[x]_{\sim}$ and $[y]_{\sim}$ are adjacent if and only if $\ann_{R}(x)\cup \ann_{R}(y)\subsetneq \ann_{R}(xy)$. For a survey relate to this topic see \cite{SurveyAG}.
\par
There are also a generalization of graphs over module. One of those graphs associated with a module, the torsion graph, denoted by $\Gamma_{R}(M)$, was introduced by P.~Malakooti~Rad, S.~Yassemi, Sh.~Ghalandarzadeh and P.~Safari in \cite{Malakooti}, with an $R$-module $M$. The set of vertices of $\Gamma_{R}(M)$ is the set of non-zero torsion elements $T(M)^{*}$ where, $T(M)^*=\{m\in M\,\mid\, \ann_{R}(m)\neq\{0\}\}$ with two distinct vertices $x$ and $y$ are adjacent if and only if $\ann_{R}(x)\cap \ann_{R}(y)\neq\{0\}$. In their work, they studied in which case $\Gamma_{R}(M)$ is connected with $\dm(\Gamma_{R}(M))$ is less than or equal to three, the relationship between the diameter of $\Gamma_{R}(M)$ and $\Gamma_{R}(R)$ and proved that the girth of $\Gamma_{R}(M)$ belongs to $\{3,\infty\}$. Note that, $\Gamma_{R}(R)$ is a special case of the torsion graph $\Gamma_{R}(M)$ when we consider $R$ as an $R$-module. For other graphs over module see also, \cite{Module1, ZeroModule, Module3, Module4, Module5}
\par
In this work, we define a new graph, which we called the compressed intersection annihilator graph $IA(R)$, as the undirected graph whose set of vertices is $Z(R/\sim)^*=Z(R/\sim)\setminus\{[0]_{\sim}\}$ and two distinct vertices $[x]_{\sim}$ and $[y]_{\sim}$ are adjacent if and only if $\ann_{R}(x)\cap \ann_{R}(y)\neq\{0\}$ for some representatives $x$ and $y$. It's known that different elements from the ring may give the same annihilator ideal. Hence we compress all vertices in $\Gamma_R(R)$ that have the same non-zero annihilator ideal in one vertex. This graph is a generalization of the graph $\Gamma_{R}(R)$. Note that if $r\in \ann_{R}(x)\cap\ann_{R}(y)$ for some $x$ and $y$, then $r\in \ann_{R}(x+y)$. Thus if $[x]_{\sim}$ is adjacent to $[y]_{\sim}$ in $IA(R)$, then $x$ is adjacent to $y$ in $T(\Gamma(R))$ for all representatives $x$ and $y$. Therefore we may consider $IA(R)$ as a way to compress the total graph $T(\Gamma(R))$. The study of this new graph help us to understand the structure of the annihilator ideals and the relation between them. The adjacency relation is always reflexive and symmetric but is not usually transitive. So the compressed intersection annihilator graph measures this lack of transitivity in which the adjacency relation is transitive if and only if the graph is complete. In the second section, we study classes of rings for which the equivalence between $Z(R)$ being an ideal and the completeness of $IA(R)$ holds. Besides, we show that if the compressed intersection annihilator graph of a ring $R$ is finite, then there exist a subring $S$ of $R$ such that $IA(S)\cong IA(R)$. Also, we show that the compressed intersection annihilator graph will never be a complete bipartite graph. In the third section,  we generalize some results from \cite{Malakooti}. In addition, we show that the graph $IA(R)$ with at least three vertices is connected and its diameter is less than or equal to three. Also, we study the relation between $\Gamma_{R}(R)$ and $IA(R)$. In the fourth section, we investigate the properties of the graph when $R=\mathbb{Z}_n$ and show that if $n$ is divisible by at least three primes, then the graph $IA(\mathbb{Z}_n)$ is connected and determine its diameter and girth. In the last section, we study the graph when $R$ is the finite direct product of integral domains with non-zero identities. This case shows an example of finite graph of an infinite ring and also, shows an example of isomorphic graphs of non-isomorphic rings. Also, we show that the graph is connected with diameter equal to two and girth equal to three when $R$ is the direct product of Artinine local rings with non-zero identities. Finally, we show that the graph $IA(R)$ is connected and not complete with diameter is less than or equal to three when $R$ is the direct product of two rings such that one of them is not an integral domain.
\par
Let $G$ be a simple undirected graph. Two different elements $x$ and $y$, from the set of vertices of $G$, are \textit{adjacent} if $\{x,\,y\}$ is an element in the set of edges of $G$ and denote it by $x-y$. A \textit{path} in a graph is a sequence of distinct vertices $x_1$, $x_2$, ..., $x_n$ such that $x_i-x_{i+1}$ for all $1\leq i<n$. The graph $G$ is \textit{connected} if there is a path between any two distinct vertices of $G$. A \textit{complete bipartite graph} is a graph whose set of vertices can be partitioned into two disjoint sets, say $A$ and $B$, in which every vertex in $A$ is adjacent to every vertex in $B$ and there is no two vertices in either $A$ or $B$ are adjacent. Its denoted by $K^{m,n}$, where $|A|=m$ and $|B|=n$. A \textit{complete graph} is a graph such that every two different vertices are adjacent and it is denoted by $K^{n}$, where $n$ is the number of vertices. On the other hand, $G$ is said to be \textit{totally disconnected} if there are no adjacent vertices. The \textit{distance}, $d(x,y)$, between two vertices $x$ and $y$ in $G$ is the length of a shortest path from $x$ to $y$ if there is a path, $d(x,x)=0$ and $d(x,y)=\infty$ otherwise. The \textit{diameter} of $G$ is defined by {$\dm(G)=\sup\{ d(x,y)\mid x\text{ and } y\text{ are vertices of } G\}$}. A \textit{cycle} is a closed path consists of more than or equal to three vertices which starts and ends at the same vertex. The \textit{girth} of $G$ is the length of a shortest cycle in $G$, it is denoted by $\gr(G)$ ($\gr(G)=\infty$ if $G$ has no cycles).
\par
An ideal $I$ of $R$ is called an \textit{annihilator ideal} if there is $a\in R$ such that $I=\ann_{R}(a)$. An element $x\in R$ is said to be a \textit{nilpotent element} if there is an integer $n\geq2$ such that $x^n=0$ and $x^{n-1}\neq 0$. Clearly any nilpotent element is a zero divisor. An element $a\in R$ is an \textit{idempotent element} if $a^2=a$. $R$ is defined to be a \textit{local ring} if it has a unique maximal ideal. $R$ is a \textit{von Neumann regular ring} if for every element $a\in R$ there exists $x\in R$ such that $a=a^{2}x$. A \textit{Noetherian ring} is a ring that satisfies the ascending chain condition (a.c.c) on ideals. This means that there is no infinite ascending sequence of ideals. $R$ is said to be an \textit{Artinian ring} if it satisfies the descending chain condition (d.c.c) on ideals. This means that there is no infinite descending sequence of ideals. Note that, any Artinian ring is a Noetherian ring but the converse is not generally true. $R$ is said to have a \textit{finite Goldie dimension} if it does not contain infinite direct sums of non-zero ideals.


\end{Section}


\begin{Section}{Compressed Intersection Annihilator Graph}
\noindent
Let $\sim$ be the multiplicative congruence relation, defined on $R$ by $x\sim y$ if and only if $\ann_{R}(x)=\ann_{R}(y)$. Let $[x_1]_{\sim}=[x_2]_{\sim}$ and $[y_1]_{\sim}=[y_2]_{\sim}$, which means that $\ann_{R}(x_1)=\ann_{R}(x_2)$ and $\ann_{R}(y_1)=\ann_{R}(y_2)$. If $r\in {{\ann_{R}(x_1)} \cap {\ann_{R}(y_1)}}$, then $r\in\ann_{R}(x_1)$ and $r\in\ann_{R}(y_1)$, it follows that $r\in\ann_{R}(x_2)$ and $r\in\ann_{R}(y_2)$, which means that $r\in \ann_{R}(x_2)\cap \ann_{R}(y_2)$. This shows that the adjacency is well defined in the following
\begin{definition}
   The compressed intersection annihilator graph is a simple undirected graph where the set of vertices is \hbox{$Z(R/\sim)^*=Z(R/\sim)\setminus\{[0]_{\sim}\}$} and the adjacency between any two different vertices $[x]_{\sim}$ and $[y]_{\sim}$, $[x]_{\sim}-[y]_{\sim}$ if and only if ${\ann_{R}(x)}\cap{\ann_{R}(y)}\neq \{0\}$ for some representative $x$ and $y$. It is denoted by $IA(R)$.
\end{definition}

Note that if $r\in \ann_{R}(x)\cap\ann_{R}(y)$ for some $x$ and $y$, then $r\in\ann_{R}(x+y)$. The converse is not always true. In fact, consider the ring $R=\mathbb{Z}_6$. We have that $4+5=3\in Z^*(R)$ and $2\in\ann_{R}(4+5)$. But $2\notin\ann_{R}(4)$ and $2\notin\ann_{R}(5)$. It follows that, if $[x]_{\sim}$ is adjacent to $[y]_{\sim}$ in $IA(R)$, then $x$ is adjacent to $y$ in $T(\Gamma(R))$ for all representations $x$ and $y$. We can find an injective maps from $V(IA(R))$ to $V(T(\Gamma(R)))$ and from $E(IA(R))$ to $E(T(\Gamma(R)))$. For example, let $f:V(IA(R))\rightarrow V(T(\Gamma(R)))$ defined by $[x]_{\sim}\mapsto x$ for some representative $x$ and define the map $g_{f}:E(IA(R))\rightarrow E(T(\Gamma(R)))$ defined by $([x]_{\sim},[y]_{\sim})\mapsto(f(x),f(y))$.

\begin{example}
 Figure~\ref{graph} represents the graph $IA(R)$, where $R=\mathbb{Z}_3\times \mathbb{Z}_3$. Figure~\ref{graph1} shows the graph $IA(R)$, for $R=\mathbb{Z}_{12}$. In the latter case, we can easily check that $\ann_{R}(2)=\ann_{R}(10)$, $\ann_{R}(3)=\ann_{R}(9)$ and $\ann_{R}(4)=\ann_{R}(8)$. Thus the set of vertices is $Z(R/\sim)^*=\{[2]_{\sim},\,[3]_{\sim},\,[4]_{\sim},\,[6]_{\sim}\}$.

\begin{figure}[H]%
 \begin{minipage}{.5\textwidth}%

  \centering
   \begin{tikzpicture}[scale=0.7,every node/.style={fill, circle,draw,scale=.5}]
    \node(n1) at (-1,1){};
    \node(n2) at (1,1){};
    \node[draw=none,rectangle, above=2mm,fill=none] (n1) at (-1,1){$[(0,1)]_{\sim}$};
    \node[draw=none,rectangle, above=2mm,fill=none] (n2) at (1,1) {$[(1,0)]_{\sim}$};
   \end{tikzpicture}

  \caption{\label{graph}$IA(R)$} {\footnotesize{$R=\mathbb{Z}_3\times \mathbb{Z}_3$}}%
 \end{minipage}%
\qquad
 \begin{minipage}{1.4in}%
   \centering

  \begin{tikzpicture}[scale=0.7,every node/.style={fill, circle,draw,scale=.5}]
    \node (n1) at (-1,-1){};
    \node (n2) at (0,1){};
    \node (n3) at (1,-1){};
    \node (n4) at (0,0){};
    \draw[-] (n1)--(n4);
    \draw[-] (n1)--(n3);
    \draw[-] (n4)--(n2);
    \draw[-] (n4)--(n3);
    \node[draw=none,rectangle, below=2mm,fill=none] (n1) at (-1,-1) {$[2]_{\sim}$};
    \node[draw=none,rectangle, above=2mm,fill=none] (n2) at (0,1)   {$[3]_{\sim}$};
    \node[draw=none,rectangle, below=2mm,fill=none] (n3) at (1,-1)  {$[4]_{\sim}$};
    \node[draw=none,rectangle, right=2mm,fill=none] (n4) at (0,0)   {$[6]_{\sim}$};
  \end{tikzpicture}
  \caption{\label{graph1}$IA(R)$}{\footnotesize{$R=\mathbb{Z}_{12}$}}%
 \end{minipage}
\end{figure}
\end{example}

We notice from the definition of the graph $IA(R)$ that it is an empty graph if and only if $Z(R)=\{0\}$. Clearly, for the extreme case, if $\ann_{R}(R)\neq\{0\}$, then $IA(R)$ is complete. But the converse is not always true.
\par
We study an essential property to a ring which is the set of zero divisors being an ideal of $R$. This give us a briefly information about the total graph $T(\Gamma(R))$ as D.~F.~Anderson and A.~Badawi in \cite{Total} breaks its study into two cases depending on whether or not $Z(R)$ is an ideal.
\par
In the next theorem, we show that if the graph $IA(R)$ is complete, then $Z(R)$ is an ideal. In theorem~\ref{thann}, corollary~\ref{finite} and theorem~\ref{Goldie}, we show that the converse is true when the set of zero divisors $Z(R)$ is an annihilator ideal, $R$ is finite or $R$ has a finite Goldie dimension.

\begin{theorem}\label{ideal}
  If the graph $IA(R)$ is complete, then $Z(R)$ is an ideal.
\end{theorem}
\begin{proof}
    Assume that $IA(R)$ is a complete graph. Let $x,\, y \in Z(R)$. We have two cases, the first one when $[x]_\sim$ and $[y]_\sim$ are two elements in $Z(R/\sim)^*$, then ${\ann_{R}(x)\cap \ann_{R}(y)}\neq{0}$. Hence there is $r\neq0$ such that $r\in{\ann_{R}(x)\cap \ann_{R}(y)}$ and so, $r(x+y)=rx+ry=0$. Therefore $x+y\in Z(R)$. In the second case, we have $[x]_{\sim}$ or $[y]_{\sim}$ is zero, and $x+y\in Z(R)$. Therefore $Z(R)$ is an ideal.
\end{proof}

The last theorem implies also from~\cite{Total} that if the graph $IA(R)$ is complete, then $T(\Gamma(R))$ is disconnected graph breaks into two components the complete induced subgraph $Z(\Gamma(R))$ of $T(\Gamma(R))$ with vertices $Z(R)$ and the induced subgraph $Reg(\Gamma(R))$ of $T(\Gamma(R))$ with vertices $Reg(R)$ the set of regular elements of $R$. Also, they characterize the graph $Reg(\Gamma(R))$ depending on whether or not $2$ is in $Z(R)$.

\begin{theorem}\label{thann}
  If $Z(R)$ has a non-zero annihilator, then $IA(R)$ is a complete graph.
\end{theorem}
\begin{proof}
  Assume that $Z(R)$ has a non-zero annihilator. Let $a\in\ann_{R}(Z(R))$ for some $a\neq0$. Let $x,\,y\in Z^{*}(R)$ be such that $[x]_\sim\neq[y]_\sim$, i.e.~$\ann_{R}(x)\neq \ann_{R}(y)$. Then $xa=0$ and $ya=0$. This means that $a\in \ann_{R}(x)\cap\ann_{R}(y)$. Thus $[x]_\sim$ is adjacent to $[y]_\sim$ which implies the completeness of $IA(R)$.
\end{proof}

Recall that from~\cite{Atiyah}, $Z(R)$ is an ideal if and only if $Z(R)$ is a prime ideal. In case, when $R$ is Artinian ring, we have the following results:
 \begin{itemize}
   \item Each prime ideal of $R$ has a non-zero annihilator (Since each prime ideal in an Artinian ring is a maximal ideal and each maximal ideal is also a minimal prime ideal. And each minimal prime ideal in a Notherian ring has a non-zero annihilator).
   \item Each non-unit of $R$ is nilpotent if and only if $R$ is local.
 \end{itemize}
  Moreover, as we know for the Artinian local ring $R$, the maximal ideal is $Z(R)$ and it is an annihilator ideal. Therefore, from the previous theorem $IA(R)$ is complete.

Theorem~\ref{Goldie} is a direct result from theorem~\ref{Z(R)} which is proved by M.~Filipowicz and M.~K\c{e}pczyk in \cite{ZeroDivisor}.

\begin{theorem}[Theorem 3.4, \cite{ZeroDivisor}]\label{Z(R)}
  If a proper ring $R$ has a finite Goldie dimension, then every finitely generated ideal of $R$, consisting of zero-divisors, has a non-zero annihilator.
\end{theorem}

One interpretation of theorem~\ref{Z(R)} is the following statement:
\begin{quote}
``If a proper ring $R$ has a finite Goldie dimension and $Z(R)$ is an ideal of $R$, then every finite set of zero divisors of $R$ has a non-zero annihilator".
\end{quote}
As a consequence of theorems~\ref{Z(R)} and \ref{ideal} and the definition of the graph, the following theorem holds.

\begin{theorem}\label{Goldie}
  Let $R$ be a ring with finite Goldie dimension. $Z(R)$ is an ideal if and only if $IA(R)$ is a complete graph.
\end{theorem}
 It is clear that, theorem~\ref{Goldie} holds for Noetherian rings as they have a finite Goldie dimension.
\par
The following theorem shows that we have only to consider Noetherian rings in order to find finite graphs.
\begin{theorem}
  If $IA(R)$ is finite graph, then there exists a Noetherian subring $S$ of $R$ such that $IA(S)\cong IA(R)$.
\end{theorem}
\begin{proof}
  Assume that $IA(R)$ is finite. Let $|V(IA(R))|=n<\infty$ and $V(IA(R))=Z(R/\sim)^*=\{[x_i]_\sim|\, x_i\in Z^*(R),\, 1\leq i\leq n\}$. Then $\ann_R(x_i)\neq \ann_R(x_j) $ for all $i\neq j$, $1\leq i,j\leq n$. Hence for every $i\neq j$, $1\leq i,j\leq n$ there exist $y_{ij}\in\ann_R(x_i)$ and $y_{ij}\notin\ann_R(x_j)$ or $y_{ij}\notin\ann_R(x_i)$ and $y_{ij}\in\ann_R(x_j)$. However, $\ann_R(y_{ij})=\ann_R(x_k)$ for some $k$. So we can choose $y_{ij}\in \{x_l|\, 1\leq l\leq n\}$. Let $S$ be the subring generated by $\{x_l|\,1\leq l\leq n\}\cup\{1\}$. Then $S$ is Noetherian. Clearly, $[x_i]_\sim=[x_j]_\sim$ if and only if $i=j$. Then $V(IA(S))=Z(S/\sim)^*=\{[x_i]_\sim|\, 1\leq i\leq n\}=V(IA(R))$. Also, for all $i$ and $j$, $\ann_S(x_i)\cap\ann_S(x_j)\neq\{0\}$ if and only if $\ann_R(x_i)\cap\ann_R(x_j)\neq\{0\}$. Therefore, $IA(S)\cong IA(R)$.
\end{proof}
The next theorem shows that the compressed intersection annihilator graph will never be a complete bipartite graph for any ring.
\begin{theorem}
  $IA(R)\neq K^{m,n}$ for all $m,n>1$.
\end{theorem}
\begin{proof}
  Suppose seeking a contradiction that $IA(R)=K^{m,n}$ for some $m,n>1$. Let $A=\{[a_i]_\sim| \, 1\leq i\leq m\}$ and $B=\{[b_i]_\sim| \, 1\leq i\leq n\}$ such that $Z(R/\sim)^*=A\cup B$ and $A\cap B=\emptyset$. Since that every element in $A$ is adjacent to every element in $B$, then $\ann_R(a_1)\cap\ann_R(b_1)\neq\{0\}$. Hence $a_1+b_1\in Z(R)$. If $a_1+b_1=0$, then $a_1=-b_1$ which implies that $\ann_R(a_1)=\ann_R(b_1)$ and this is a contradiction. It follows that $[a_1+b_1]_\sim\in Z(R/\sim)^*$. Then $[a_1+b_1]_\sim\in A$ or $[a_1+b_1]_\sim\in B$. Assume that $[a_1+b_1]_\sim\in A$. Then $[a_1+b_1]_\sim=[a_i]_\sim$ for some $i$. Since $IA(R)$ is complete bipartite graph, then there is $r\neq0$ such that $r\in \ann_R(b_1)\cap\ann_R(a_1)$. Thus $r\in\ann_R(a_i)$. Therefore $i=1$ which means that $[a_1+b_1]_\sim=[a_1]_\sim$. Since there exist $s\neq0$ such that $s\in\ann_R(b_2)\cap\ann_R(a_1)$, then $sb_2=0$ and $0=sa_1=s(a_1+b_1)=sa_1+sb_1=sb_1$. Therefore $[b_1]_\sim$ is adjacent to $[b_2]_\sim$ and this is a contradiction. Similarly, if $[a_1+b_1]_\sim\in B$. So $IA(R)\neq K^{m,n}$ for all $m,n>1$.
\end{proof}

\end{Section}

\begin{Section}{Diameter and Girth}
 In this section, we determine the diameter of the graph and its girth.
 Now, the following theorems~\ref{th1}, \ref{th2} and corollary~\ref{grth} are generalizations of theorems~$3.1$, $4.1$ and corollary~$4.2$ from \cite{Malakooti} respectively, if we consider $R$ as an $R$-module in the graph $\Gamma_{R}(R)$. Theorem \ref{th1} shows that if either the ring $R$ is a von Neumann regular ring and $R\ncong \ann_{R}(x)\oplus \ann_{R}(y)$ for any two distinct $x,\,y\in Z^{*}(R)$ or $\nl(R)\neq \{0\}$, then $IA(R)$ is connected with a diameter less than or equal to three. The proof of the theorem depends on an equivalent definition of the von Neumann regular ring that is every principal ideal is generated by an idempotent element.
 \par
 The proofs of the following theorems~\ref{th1}, \ref{th2} and corollary~\ref{grth} are analogues to the proofs of parts $2$, $3$ of theorem~$3.1$, theorem~$4.1$ and corollary~$4.2$ of \cite{Malakooti} respectively. We have just to replace $T(M)^*$ by $Z(R/\sim)^*$, $Ann(x)$ by $\ann_{R}(x)$ and the vertices in $\Gamma_{R}(M)$ by the vertices in $IA(R)$.

\begin{theorem}\label{th1}
  $IA(R)$ is connected with $\dm(IA(R))\leq 3$ if one of the following conditions holds:
    \begin{enumerate}
      \item {$R$ is a von Neumann regular ring and $R\ncong \ann_{R}(x)\oplus \ann_{R}(y)$ for any two distinct $x,\,y\in Z^{*}(R)$.}
      \item {$\nl(R)\neq \{0\}$ (i.e.~$R$ is not reduced).}
    \end{enumerate}
\end{theorem}

The following theorem shows that the girth of the graph $IA(R)$ belongs to $\{3,\,\infty\}$. In corollary~\ref{grth}, we show that if $IA(R)$ is a connected graph with more than or equal to three vertices, then $IA(R)$ contains a cycle.

\begin{theorem}\label{th2}
  If $IA(R)$ contains a cycle, then $\gr(IA(R))=3$.
\end{theorem}

\begin{corollary}\label{grth}
  If $IA(R)$ is a connected graph with $\mid Z(R/\sim)^*\mid >2$, then $IA(R)$ contains a cycle and $\gr(IA(R))=3$.
\end{corollary}

The following theorem shows that the graph $IA(R)$ with at least three vertices is connected, its diameter is less than or equal to three and from the previous corollary, it follows its girth is equal to three.

\begin{theorem}\label{diam}
  Let $\mid Z(R/\sim)^*\mid >2$. Then $IA(R)$ is connected and $\dm(IA(R))\leq3$.
\end{theorem}
\begin{proof}
  Let $[x]_\sim\in Z(R/\sim)^*$ and $[y]_{\sim}\in Z(R/\sim)^*$ be two non-equal nonadjacent vertices. Let $[z]_\sim\in Z(R/\sim)^*$ such that $[x]_\sim\neq[z]_\sim\neq [y]_\sim$. If $[z]_\sim$ is adjacent to both $[x]_\sim$ and $[y]_\sim$, then $d(x,y)=2$. If $[z]_\sim$ is not adjacent to $[y]_\sim$ and it is adjacent to $[x]_\sim$, we have two subcases. The first subcase if there is $r\in\ann(x)$ and $r\notin\ann(z)$ (since $[x]_\sim\neq[z]_\sim$), then there is $t\in Z^*(R)$ such that $tz=0$, $tx=0$ and $ty\neq0$. Since $ry\neq0$, hence ${[x]_\sim}-{[rz]_\sim}-{[ry]_\sim}-{[y]_\sim}$ is a path between $[x]_\sim$ and $[y]_\sim$. Then $d(x,y)\leq3$. Similarly, the second subcase if there is $r\in\ann(z)$ and $r\notin\ann(x)$ (since $[x]_\sim\neq[z]_\sim$), then $d(x,y)\leq3$. Similarly to the previous case, if $[z]_\sim$ is adjacent to $[y]_\sim$ and it is not adjacent to $[x]_\sim$, hence $d(x,y)\leq3$. If $[z]_\sim$ is neither adjacent to $[x]_\sim$ nor $[y]_\sim$, then there is $t\in Z^*(R)$ such that $tz=0$, $tx\neq0$ and $ty\neq0$. So ${[x]_\sim}-{[tx]_\sim}-{[ty]_\sim}-{[y]_\sim}$ is a path between $[x]_\sim$ and $[y]_\sim$. Thus $d(x,y)\leq3$. Therefore, $IA(R)$ is connected and $\dm(IA(R))\leq3$.
\end{proof}

\begin{remark}\label{card 2}\leavevmode
   \begin{itemize}
     \item {If $\mid Z(R/\sim)^*\mid=2$, then we have two cases. First one, if $Z(R)$ is not an ideal, then from theorem~\ref{ideal}, $IA(R)$ is not a complete graph which means that $IA(R)$ is totally disconnected graph. Second one, if $Z(R)$ is an ideal, suppose that $Z(R/\sim)^*=\{[a]_\sim,\,[b]_\sim\}$. Then $a+b\in Z(R)$. If $[a+b]_\sim=[0]_\sim$, then $a+b=0$. Then $a=-b$. Whence $\ann_R(a)=\ann_R(b)$ and this is a contradiction. Therefore $[a+b]_\sim=[a]_\sim$ or $[a+b]_\sim=[b]_\sim$. If $[a+b]_\sim=[a]_\sim$, then for all $0\neq r\in\ann_R(a)$, $r\in\ann_R(a+b)$. Which means that $0=ra=r(a+b)=ra+rb=rb$. Thus $r\in \ann_R(b)$. Similarly, if $[a+b]_\sim=[b]_\sim$, then $IA(R)$ is complete.}
     \item {If $\mid Z(R/\sim)^*\mid=3$, then $IA(R)$ is connected. Hence $\gr(IA(R))=3$. It follows that $IA(R)$ is complete and $\dm(IA(R))=1$.}
   \end{itemize}
\end{remark}

The following theorem shows that $\Gamma_{R}(R)$ is complete if and only if $IA(R)$ is complete. Theorem~\ref{conn} shows that $\Gamma_{R}(R)$ is connected if and only if $IA(R)$ is connected. Moreover, they have the same diameters.

\begin{theorem}
  $\Gamma_{R}(R)$ is complete if and only if $IA(R)$ is complete.
\end{theorem}

\begin{proof}
   Assume that $\Gamma_{R}(R)$ is complete. Let $[x]_{\sim},\, [y]_{\sim}\in Z(R/\sim)^{*}$ be two distinct vertices. Then $x,\,y\in Z^{*}(R)=T(R)^{*}$ for any representative $x,\,y$. Hence there is $r\in \ann_{R}(x)\cap \ann_{R}(y)$. Therefore $IA(R)$ is complete.\\
  For the converse, let $x,\,y\in T(R)^*=Z^*(R)$. If $\ann_{R}(x)=\ann_{R}(y)$, then $x-y$. If $\ann_{R}(x)\neq \ann_{R}(y)$, then $[x]_{\sim}$ and $[y]_{\sim}$ are two distinct vertices in $IA(R)$. Since $IA(R)$ is complete, then there is $r\in Z^{*}(R)$ such that $r\in \ann_{R}(x)\cap \ann_{R}(y)$. So $x-y$. Therefore $\Gamma_{R}(R)$ is complete.
\end{proof}

\begin{theorem}\label{conn}
  $\Gamma_{R}(R)$ is connected if and only if $IA(R)$ is connected. Moreover, if $|Z(R/\sim)^*|>1$, then $\dm(\Gamma_{R}(R))=\dm(IA(R))$.
\end{theorem}
\begin{proof}
  Assume that $\Gamma_{R}(R)$ is connected. Let $[x]_{\sim},\, [y]_{\sim}\in Z(R/\sim)^*$ be two distinct vertices in $IA(R)$. Then there is a path between $x$ and $y$. Let ${{r_0}={x}}-{r_1}-....-{r_{n-1}}-{{y}={r_n}}$ be the shortest path of length $n$ between $x$ and $y$, i.e.~$d(x,y)=n$. If $\ann_{R}(r_i)=\ann_{R}(r_{i+1})$ for some $i\geq1$, and since $\ann_{R}(r_{i-1})\cap \ann_{R}(r_i)\neq \{0\}$, then $r_{i-1}-r_{i+1}$. We can collapse the path into {${{r_0}={x}}-{r_1}-...-{r_{i-1}}-{r_{i+1}}-...-{r_{n-1}}-{{y}={r_n}}$}, i.e.~$d(x,y)<n$, which is a contradiction. Therefore $\ann_{R}(r_i)\neq \ann_{R}(r_{i+1})$ for all $0\leq i\leq {n-1}$. Thus $[r_i]_{\sim}\neq [r_{i+1}]_{\sim}$ and $\ann_{R}(r_i)\cap \ann_{R}(r_{i+1})\neq\{0\}$ for all $0\leq i\leq {n-1}$. It follows that {${[r_0]_{\sim}}={[x]_{\sim}}-{[r_1]_{\sim}}-....-{[r_{n-1}]_{\sim}}-{{[y]_{\sim}}={[r_n]_{\sim}}}$} is a path in $IA(R)$ between $[x]_{\sim}$ and $[y]_{\sim}$. Hence $IA(R)$ is connected and $d([x]_{\sim},[y]_{\sim})\leq n=d(x,y)$. So $\dm(IA(R))\leq\dm(\Gamma_{R}(R))$.\\
  For the converse, assume that $IA(R)$ is connected. Let $x,\,y\in T(R)^*=Z^{*}(R)$. If $\ann_{R}(x)=\ann_{R}(y)$ i.e.~$[x]_{\sim}=[y]_{\sim}$, then $x\,-\,y$ in the graph $\Gamma_{R}(R)$ which means that $d(x,y)=1$. If $\ann_{R}(x)\neq \ann_{R}(y)$, i.e.~$[x]_{\sim}\neq [y]_{\sim}$, then there is a path between $[x]_{\sim}$ and $[y]_{\sim}$. Let ${{[s_0]_{\sim}}={[x]_{\sim}}}-{[s_1]_{\sim}}-....-{[s_{m-1}]_{\sim}}-{{[y]_{\sim}}={[s_m]_{\sim}}}$ be the shortest path of length $m$ between $[x]_{\sim}$ and $[y]_{\sim}$, i.e.~$d([x]_{\sim}, [y]_{\sim})=m$. Thus $\ann_{R}(s_i)\neq \ann_{R}(s_{i+1})$ and $\ann_{R}(s_i)\cap \ann_{R}(s_{i+1})\neq \{0\}$ for all $0\leq i\leq{m-1}$. Therefore ${{s_0}={x}}-{s_1}-...-{s_{m-1}}-{{y}={s_m}}$ is a path in $\Gamma_{R}(R)$ between $x$ and $y$. Then $\Gamma_{R}(R)$ is connected and $d(x,y)\leq m=d([x]_{\sim},[y]_{\sim})$. It follows that $\dm(\Gamma_{R}(R))\leq\dm(IA(R))$.\\
  From the two parts of the proof, it follows that if $|Z(R/\sim)^*|>1$, then $\dm(IA(R))=\dm(\Gamma_{R}(R))$.
\end{proof}

\end{Section}


\begin{Section}{Compressed Intersection Annihilator Graph of $\mathbb{Z}_n$}
\noindent
 Throughout this section we assume that $R=\mathbb{Z}_n$ for some integer $n>1$. The following two lemmas identify which elements in $\mathbb{Z}_n$ are zero divisors and when they have the same annihilators.

\begin{lemma}\label{zeros}
  Let $k$ and $n$ be integers with $1<k<n$. Then $k\in Z^*(\mathbb{Z}_n)$ if and only if $g.c.d(k,n)\neq 1$.
\end{lemma}
\begin{proof}
  Let $k\in Z^{*}(\mathbb{Z}_n)$. Then there is $l\in Z^{*}(\mathbb{Z}_n)$, $0\neq l<n$ such that $kl=0\, (\text{mod}\,n)$. Hence for some positive integer $m$, $kl=mn$. If $g.c.d(k,n)=1$, so $k/m$ and also $n/l$ and this is a contradiction with $l<n$.\\
  For the converse, assume that $g.c.d(k,n)=r\neq1$. Then there are integers $l\neq0$, $m\neq 0$ such that $n=rl$ and $k=rm$. Hence $kl=rml=nm$, so $kl=0 \,(\text{mod}\,n)$. Therefore $k\in Z^{*}(\mathbb{Z}_n)$.
\end{proof}

\begin{lemma}\label{ann}
Let $k\in \mathbb{Z}_n$ and $k\neq0$. If $g.c.d(k,n)=l$ with $l\neq1$, then $\ann_{R}(k)=\ann_{R}(l)$.
\end{lemma}
\begin{proof}
  Let $g.c.d(k,n)=l$, $l\neq1$. Then there exist non-zero $r,\,s\in\mathds{N}$ such that $k=rl$, $n=sl$ with $g.c.d(r,s)=1$. It is clear that, $\ann_{R}(l)\subset \ann_{R}(k)$. Now, we want to show that $\ann_{R}(k)\subset \ann_{R}(l)$. Let $h\in \ann_{R}(k)$, $h\neq0$. Then $hk=qn$ for non-zero integer $q\in\mathds{N}$. By substitution, $hrl=qsl$. By cancellation of $l\neq0$, we have $hr=qs$ but $g.c.d(r,s)=1$. Therefore $s$ must divide $h$, then $sl$ must divide $hl$, which means that $hl$ is a multiple of $n$. Thus $hl=0\,(\text{mod}\,n)$ and therefore $h\in \ann_{R}(l)$, and $\ann_{R}(k)\subset \ann_{R}(l)$. Thus $\ann_{R}(k)=\ann_{R}(l)$.
\end{proof}

\begin{corollary}\label{vertices}
  Let $[k]_{\sim}\in{Z({\mathbb{Z}_n}/\sim)^*}$. Then for some representative $l$, $[k]_{\sim}=[l]_{\sim}$ and $l/n$.
\end{corollary}
\begin{proof}
  From lemma~\ref{zeros}, $k\in {Z^{*}(\mathbb{Z}_n)}$ implies that $g.c.d(k,n)\neq1$ say, $g.c.d(k,n)=l\neq1$, so from lemma~\ref{ann}, $\ann_{R}(k)=\ann_{R}(l)$ which means that $[k]_{\sim}=[l]_{\sim}$ and $l/n$ as required.
\end{proof}
In the previous corollary, we identify the vertices $Z(R/\sim)^*$ of $IA(R)$ in view of properties of $\mathbb{Z}_n$ and we use it in proving the following two lemmas. In the next lemma, we determine when the vertices of the graph $IA(R)$ are adjacent.
\begin{lemma}\label{adj}
  Let ${[k]_{\sim},\, [l]_{\sim}}\in Z(R/\sim)^*$ such that $[k]_{\sim}\neq[l]_{\sim}$. Then $g.c.d(k,l)\neq1$ if and only if $[k]_{\sim}$ is adjacent to $[l]_{\sim}$.
\end{lemma}
\begin{proof}
  Let ${[k]_{\sim},\, [l]_{\sim}}\in Z(R/\sim)^*$ where, $[k]_{\sim}\neq[l]_{\sim}$ and without any loss of generality, we assume $k/n$ and $l/n$.\\
  Assume that $g.c.d(k,l)=r$ with $r\neq1$. Then there are non-zero $q,\,p\in \mathds{N}$ such that $k=qr$ and $l=pr$.
  Since $k/n$ and $l/n$, so that $r/n$, then there is a positive integer $m$ such that $n=mr$. By substitution,
  \begin{eqnarray*}
    mk &=& mqr = qn \\
       &=& 0 \,(\text{mod}\,n) \\
    ml &=& mpr = pn\\
       &=& 0 \,(\text{mod}\, n)
  \end{eqnarray*}
  i.e.~$m\in {\ann_{R}(k)\cap \ann_{R}(l)}$. Therefore $[k]_{\sim}$ is adjacent to $[l]_{\sim}$.\\
  For the converse, assume that $[k]_{\sim}$ is adjacent to $[l]_{\sim}$. Then there is $h\in Z^{*}(\mathbb{Z}_n)$ such that $h\in \ann_{R}(k)\cap \ann_{R}(l)$. Thus from corollary~\ref{vertices} there is $m\in Z^{*}(\mathbb{Z}_n)$ such that $[h]_{\sim}=[m]_{\sim}$ and $m/n$ which implies that $m\in \ann_{R}(k)\cap \ann_{R}(l)$. Hence $mk=0\,(\text{mod}\,n)$ and $ml=0\,(\text{mod}\,n)$. Then there are positive integers $p,\,q$ such that $mk=pn$ and $ml=qn$. Since $m/n$, then there is positive integer $r\neq1$ such that $n=rm$. It follows that $mk=prm$ and $ml=qrm$. By cancellation of $m\neq0$; $k=pr$, $l=qr$ and $r\neq1$. Therefore $g.c.d(k,l)\neq1$.
\end{proof}
In the next lemma, we show that if $n$ is divisible by at least three primes, then the graph $IA(\mathbb{Z}_n)$ is always connected and its diameter is less than or equal to two and therefore it has a cycle of length three.
\begin{lemma}\label{thm2}
  Let $n=\prod_{i=1}^{m}{p_i}$, where $m\geq3$ and $p_i$'s are prime numbers. Then $IA(R)$ is connected graph with $\dm(IA(R))\leq2$ and $\gr(IA(R))=3$
\end{lemma}
\begin{proof}
  Let $[v_1]_{\sim},\,[v_2]_{\sim}\in Z^{*}(R/\sim)$ such that $[v_1]_{\sim}\neq[v_2]_{\sim}$. Then by corollary~\ref{vertices}, there exist $k,\,l\in Z^{*}(R)$ such that $[k]_{\sim}=[v_1]_{\sim}$ and $[l]_{\sim}=[v_2]_{\sim}$ with $k/n$ and $l/n$. We have two cases, if $g.c.d(k,\,l)\neq1$, then by lemma~\ref{adj}, $[k]_{\sim}$ is adjacent to $[l]_{\sim}$. It follows that $d(k,\,l)=d(v_1,\,v_2)=1$. If $g.c.d(k,\,l)=1$, let $k=\prod_{i=1}^{r}{p_{t_i}}$ and $l=\prod_{i=1}^{s}{p_{j_i}}$, where the set $\{p_{t_i}\}_{i=1}^{r}\subset \{p_{i}\}_{i=1}^{m}$ is distinct from the set $\{p_{j_i}\}_{i=1}^{s}\subset \{p_{i}\}_{i=1}^{m}$ and $r,\,s < m$. Let $u=p_{j_f}p_{t_g}$ for some $p_{j_f}\in \{p_{j_i}\}_{i=1}^{s}$ and $p_{t_g}\in\{p_{t_i}\}_{i=1}^{r}$. Then $[u]_{\sim}\neq[k]_{\sim}$, $[u]_{\sim}\neq[l]_{\sim}$ and $[u]_{\sim}\neq[0]_{\sim}$. By lemma~\ref{adj}, $[k]_{\sim}-[u]_{\sim}-[l]_{\sim}$. Thus $d(k,l)=d(v_1,\,v_2)=2$. It follows that $IA(R)$ is connected and $\dm(IA(R))\leq2$. From corollary~\ref{grth}, $\gr(IA(R))=3$.
\end{proof}
\begin{example}\label{example}{In the following, let $p,\,q$ and $r$ be distinct prime numbers. Then we have the following three cases:
 \begin{enumerate}
  \item {Let $R=\mathbb{Z}_{pq}$. From corollary~\ref{vertices}, ${Z(R/\sim)^*}=\{[p]_{\sim},\, [q]_{\sim}\}$. Since $g.c.d(p,q)=1$, then from Lemma~\ref{adj}, the graph $IA(R)$ is a totally disconnected graph as shown in figure~\ref{graph2}.}
     \begin{figure}[H]%
      \centering
      \begin{tikzpicture}[scale=0.7,every node/.style={fill, circle,draw,scale=.5}]
      \node(n1) at (-1,1){};
      \node(n2) at (1,1){};
      \node[draw=none,rectangle, above=3mm,fill=none] (n1) at (-1,1){$[p]_{\sim}$};
      \node[draw=none,rectangle, above=3mm,fill=none] (n2) at (1,1) {$[q]_{\sim}$};
      \end{tikzpicture}
      \caption{\label{graph2} $IA(R)$} {\footnotesize{$R=\mathbb{Z}_{pq}$}}%
     \end{figure}

  \item {Let $R=\mathbb{Z}_{pqr}$. From corollary~\ref{vertices}, the set of vertices is $${Z(R/\sim)^*}=\{[p]_{\sim},\,[q]_{\sim},\,[r]_{\sim},\,[pq]_{\sim},\,[pr]_{\sim},\,[qr]_{\sim}\}.$$ Then the graph is connected, $\textup{diam}(IA(R))=2$ and $\textup{gr}(IA(R))=3$ as shown in figure~\ref{graph3}.}
     \begin{figure}[H]%
     \centering
     \begin{tikzpicture}[scale=0.7,every node/.style={fill, circle,draw,scale=.5}]
      \node(n1) at (-1,1){};
      \node(n2) at (1,1){};
      \node(n3) at (3,1){};
      \node(n4) at (-1,-1){};
      \node(n5) at (3,-1){};
      \node(n6) at (1,-3){};
      \draw[-] (n1)--(n2);
      \draw[-] (n1)--(n4);
      \draw[-] (n2)--(n3);
      \draw[-] (n2)--(n4);
      \draw[-] (n2)--(n5);
      \draw[-] (n3)--(n5);
      \draw[-] (n4)--(n5);
      \draw[-] (n4)--(n6);
      \draw[-] (n5)--(n6);
      \node[draw=none,rectangle, left=3mm,fill=none] (n1) at (-1,1){$[p]_{\sim}$};
      \node[draw=none,rectangle, above=3mm,fill=none] (n2) at (1,1) {$[pr]_{\sim}$};
      \node[draw=none,rectangle, right=3mm,fill=none] (n3) at (3,1) {$[r]_{\sim}$};
      \node[draw=none,rectangle, left=3mm,fill=none] (n4) at (-1,-1){$[pq]_{\sim}$};
      \node[draw=none,rectangle, right=3mm,fill=none] (n5) at (3,-1) {$[qr]_{\sim}$};
      \node[draw=none,rectangle, below=3mm,fill=none] (n6) at (1,-3){$[q]_{\sim}$};
      \end{tikzpicture}
      \caption{\label{graph3}$IA(R)$} {\footnotesize{$R=\mathbb{Z}_{pqr}$}}%
     \end{figure}

  \item\label{3} {Let $R=\mathbb{Z}_{p^m}$ for a positive integer $m$. In that case $Z(R/\sim)^*=\{[p^i]_{\sim}|1\leq i<m\}$ and all vertices are adjacent, so it will be a complete graph with $(m-1)$-vertices. This means that $IA(R)=K_{m-1}$.}
 \end{enumerate}}

\end{example}
\begin{remark}\leavevmode
   \begin{itemize}
     \item  {Under the condition of lemma~\ref{thm2}, if there is at least one distinct prime, then the graph is connected and $\dm(IA(R))=2$. If there is no distinct primes, then this is the case in part~\ref{3} of example~\ref{example}.}
     \item  {From \ref{3} in example~\ref{example}, all complete graphs of $m$-vertices may be realized as $IA(R)=K_m$. For instance, we can take $R=\mathbb{Z}_{p^{m+1}}$.}
   \end{itemize}
\end{remark}
\end{Section}
%
%
\begin{Section}{Compressed Intersection Annihilator Graph of a Finite Product of Rings}
\noindent
In this section, we study the properties of the graph $IA(R)$ when the ring $R$ is one of the following:
\begin{itemize}
  \item {The finite direct product of two or more integral domains with non-zero identities.}
  \item {The finite product of Artinian local rings.}
  \item {The direct product of two rings such that one of them is not an integral domain.}
\end{itemize}
Now, we investigate the case when $R$ is the finite direct product of two or more integral domains with non-zero identities. We notice that the graph on an integral domain is empty.

\begin{theorem}\label{thm1}
  Let $R=A\times B$ where $A$ and $B$ are integral domains with non-zero identities. Then $IA(R)$ is a totally disconnected graph.
\end{theorem}
\begin{proof}
 We need to determine the set of vertices $Z(R/\sim)^*$. Let $(x,\,y)\in Z^{*}(R)$, then there is $(h,\,k)\in Z^{*}(R)$ such that $(h,\,k)(x,\,y)=(0_A,\,0_B)$. Hence $hx=0_A$ and $ky=0_B$ and since $A,\,B$ are integral domains, then $x=0_A$ or $h=0_A$ and $y=0_B$ or $k=0_B$. Therefore, the set of zero divisors without the zero element $(0_A,\,0_B)$ may be partitioned into two disjoint sets, $V_A=\{(a,0_B)\in R\,|\, a\in A\setminus\{0_A\}\}$ and $V_B=\{(0_A,b)\in R\,|\, b\in B\setminus\{0_B\}\}$. Thus $Z^{*}(R)=V_A\cup V_B$. But we can easily show that $\ann_{R}(u)=V_B$ for all $u\in V_A$ and similarly, $\ann_{R}(v)=V_A$ for all $v\in V_B$. Therefore $V_A\subseteq [(1_A,\,0_B)]_{\sim}$ and $V_B\subseteq [(0_A,\,1_B)]_{\sim}$. Thus $Z(R/{\sim})^*=\{[(1_A,\,0_B)]_{\sim},\,[(0_A,\,1_B)]_{\sim}\}$. Since $|Z(R/{\sim})^*|=2$ and $Z(R)$ is not an ideal. Thus from remark~\ref{card 2}, $IA(R)$ is totally disconnected graph with two vertices.
\end{proof}
\begin{example}
  {Let $R=\mathbb{Z}_p\times \mathbb{Z}_q$, where $p$ and $q$ are any two primes. Then $Z(R/\sim)^*=\{[(1,0)]_{\sim},\,[(0,\,1)]_{\sim}\}$ and the graph $IA(R)$ is a disconnected graph as shown in figure~\ref{graph4}.
  \begin{figure}[H]%
      \centering%
  \[    \begin{tikzpicture}[scale=0.7,every node/.style={fill, circle,draw,scale=.5}]%
      \node(n1) at (-1,1){};
      \node(n2) at (1,1){};
      \node[draw=none,rectangle, above=3mm,fill=none] (n1) at (-1,1){$[(1,\,0)]_{\sim}$};
      \node[draw=none,rectangle, above=3mm,fill=none] (n2) at (1,1) {$[(0,\,1)]_{\sim}$};
      \end{tikzpicture}%
\]
      \caption{\label{graph4}$IA(R)$} {\footnotesize{$R=\mathbb{Z}_p\times \mathbb{Z}_q$}}%
     \end{figure}}%
\end{example}
\begin{theorem}\label{thm3}
  Let $\{A_i\}_{i=1}^{n}$ be a set of integral domains with non-zero identities with an integer $n>2$, and let $R=\prod_{i=1}^{n}{A_i}$. Then $IA(R)$ is connected with $\dm(IA(R))=2$ and $\gr(IA(R))=3$.
\end{theorem}
\begin{proof}
  Let $X,\,Y\in Z^{*}(R)$. Then $X=(a_1,\,a_2,\,...,\,a_n)$ and $Y=(b_1,\,b_2,\,...,\,b_n)$, where $a_i,\,b_i\in A_i$ with $a_k=0_{A_{k}}$ and $b_l=0_{A_l}$ for some $1\leq k,\,l\leq n$. Let $[X]_{\sim}\neq[Y]_{\sim}$ and $\ann_{R}(X)\cap \ann_{R}(Y)=\{0\}$. Let $Z=(c_1,\,c_2,\,...,\,c_n)\in Z^{*}(R)$ with $c_k=0_{A_k}$ and $c_l=0_{A_l}$. Define for any $1\leq j\leq n$, $I_{A_j}=(u_1,\,u_2,\,...,\,u_n)$ with $u_j=1_{A_j}$ and $u_i=0_{A_i}$ for all $1\leq i\leq n$ and $i\neq j$. Then we can easily check that $I_{A_k}X=0$ and $I_{A_k}Z=0$ also, $I_{A_l}Y=0$ and $I_{A_l}Z=0$, so that, there is an edge between $[X]_{\sim}$ and $[Z]_{\sim}$ and an edge between $[Z]_{\sim}$ and $[Y]_{\sim}$. Therefore there is a path between $[X]_{\sim}$ and $[Y]_{\sim}$. Thus $IA(R)$ is connected. Moreover, $\dm(IA(R))\leq2$. There are vertices which are not adjacent. For, $W=(1_{A_1},\,1_{A_2},\,....\,1_{A_{n-1}},\,0_{A_n})\in R$ and $V=(0_{A_1},\,1_{A_2},\,...,\,1_{A_n})\in R$ to find an element to annihilate $V$ and $W$ together it must be $0=(0_{A_1},\,...,\,0_{A_n})\in R$, it follows that $\dm(IA(R))=2$. From corollary \ref{grth}, $\gr(IA(R))=3$.
\end{proof}
%
%
Since the only fact that we used in the proofs of theorems~\ref{thm1} and~\ref{thm3} is that $A_i$ for $1\leq i\leq n$ are integral domains, so the different graphs constructed for a given $n$ would be isomorphic.

\begin{example}\label{2}{Graph \ref{graph6} represents the following two cases:
  \begin{enumerate}
    \item {Let $p,\,q,\,r$ be three prime numbers and $R=\mathbb{Z}_p\times \mathbb{Z}_q\times \mathbb{Z}_r$. Then
    \begin{eqnarray}
    \nonumber
      Z(R/\sim)^*&=&\{[(0,\,0,\,1)]_{\sim},\,[(0,\,1,\,0)]_{\sim},\,[(1,\,0,\,0)]_{\sim},\,[(0,\,1,\,1)]_{\sim},\,[(1,\,0,\,1)]_{\sim}, \\
    \nonumber    && [(1,\,1,\,0)]_{\sim}\}.
    \end{eqnarray}}
    \item {\label{4}Let $R=\mathbb{Z}\times \mathbb{Z}\times \mathbb{Z}$. Then the set of vertices is
    \begin{eqnarray}
    \nonumber
      Z(R/\sim)^*&=&\{[(0,\,0,\,1)]_{\sim},\,[(0,\,1,\,0)]_{\sim},\,[(1,\,0,\,0)]_{\sim},\,[(0,\,1,\,1)]_{\sim},\,[(1,\,0,\,1)]_{\sim}, \\
    \nonumber    && [(1,\,1,\,0)]_{\sim}\}.
    \end{eqnarray}}
  \end{enumerate}
  \begin{figure}[H]%
     \centering%
     \begin{tikzpicture}[scale=0.7,every node/.style={fill, circle,draw,scale=.5}]%
      \node(n1) at (-1,1){};
      \node(n2) at (1,1){};
      \node(n3) at (3,1){};
      \node(n4) at (-1,-1){};
      \node(n5) at (3,-1){};
      \node(n6) at (1,-3){};
      \draw[-] (n1)--(n2);
      \draw[-] (n1)--(n4);
      \draw[-] (n2)--(n3);
      \draw[-] (n2)--(n4);
      \draw[-] (n2)--(n5);
      \draw[-] (n3)--(n5);
      \draw[-] (n4)--(n6);
      \draw[-] (n4)--(n5);
      \draw[-] (n5)--(n6);
      \node[draw=none,rectangle, left=3mm,fill=none] (n1) at (-1,1) {$[(0,\,1,\,1)]_{\sim}$};
      \node[draw=none,rectangle, above=3mm,fill=none](n2) at (1,1)  {$[(0,\,1,\,0)]_{\sim}$};
      \node[draw=none,rectangle, right=3mm,fill=none](n3) at (3,1)  {$[(1,\,1,\,0)]_{\sim}$};
      \node[draw=none,rectangle, left=3mm,fill=none] (n4) at (-1,-1){$[(0,\,0,\,1)]_{\sim}$};
      \node[draw=none,rectangle, right=3mm,fill=none](n5) at (3,-1) {$[(1,\,0,\,0)]_{\sim}$};
      \node[draw=none,rectangle, below=3mm,fill=none](n6) at (1,-3) {$[(1,\,0,\,1)]_{\sim}$};
      \end{tikzpicture}%
      \caption{\label{graph6}$IA(R)$} {\footnotesize{$R=\mathbb{Z}_p\times \mathbb{Z}_q\times \mathbb{Z}_r$  \&\\ $R=\mathbb{Z}\times \mathbb{Z}\times \mathbb{Z}$}}%
     \end{figure}}%
\end{example}
\begin{remark}\leavevmode
  \begin{itemize}
    \item {From the two parts in example~\ref{2}, we notice that we may have an isomorphic graphs for non-isomorphic rings and from part~\ref{4}, we may have a finite graph for an infinite ring.}
    \item {Let $p$, $q$ and $r$ be three distinct primes. The graph $IA(R)$ when $R=\mathbb{Z}_p\times \mathbb{Z}_q\times \mathbb{Z}_r$ is isomorphic to the graph $IA(R)$ when $R=\mathbb{Z}_{pqr}$ as shown in figures~\ref{graph3} and \ref{graph6}. The condition that $p$, $q$, and $r$ are distinct is essential, since for example, $IA(\mathbb{Z}_2\times \mathbb{Z}_2)$ is not isomorphic to $IA(\mathbb{Z}_4)$.}
  \end{itemize}
\end{remark}

Now, we investigate the properties of the compressed intersection annihilator graph $IA(R)$, when the ring $R$ is a finite direct product of Artinian local rings with non-zero identities. Notice that the graph on an Artinian local ring is complete since the set of zero divisors is an annihilator ideal. In the next theorem, we show that the graph $IA(R)$ is connected and calculate its diameter and girth.

\begin{theorem}\label{thm4}
  Let $\{R_i\}_{i=1}^{n}$ be a set of Artinian local rings with non-zero identities with an integer $n\geq2$, and let $M_i=Z(R_i)=\ann_{R}(a_i)$ be the unique maximal ideal of $R_i$ for each $1\leq i\leq n$ for some nilpotent element $a_i$, and let $R=\prod_{i=1}^{n}{R_i}$. Then $IA(R)$ is a connected graph with $\dm(IA(R))=2$ and $\gr(IA(R))=3$.
\end{theorem}
\begin{proof}
  Let $X,\, Y\in R$ be two non-zero zero-divisors such that $[X]_{\sim}\neq[Y]_{\sim}$, then for each $1\leq i\leq n$, there exist $x_i,\,y_i\in R_i$, such that $X=(x_1,\,...,\,x_n)$ and $Y=(y_1,\,...,\,y_n)$ with $x_k\in M_k$ and $y_l\in M_l$ for some $1\leq k,\,l\leq n$. Let $Z=(z_1,\,...,\,z_n)$ with $z_k=a_k$, $z_l=a_l$ and $z_i\in R_i$ for all $i\neq k,\,l$. Define  $I_{R_j}=(u_1,\,...,\,u_n)$ where $u_j=a_j$ and $u_i=0_{R_i}$ for all $i\neq j$. By properties of nilpotent elements and $M_i=\ann_{R}(a_i)$, we can easily show that $I_{R_k}X=0$  and $I_{R_k}Y=0$, and also, $I_{R_l}Z=0$ and $I_{R_l}Y=0$. It follows that, there is an edge between $[X]_{\sim}$ and $[Z]_{\sim}$ and thus, there is an edge between $[Z]_{\sim}$ and $[Y]_{\sim}$. Therefore there is a path between $[X]_{\sim}$ and $[Y]_{\sim}$. This means that $IA(R)$ is connected and also $\dm(IA(R))\leq 2$. There are vertices which are not adjacent. For instance, $W=(1_{R_1},\,1_{R_2},\,....\,1_{R_{n-1}},\,0_{R_n})\in R$ and $V=(0_{R_1},\,1_{R_2},\,...,\,1_{R_n})\in R$ to find an element to annihilate $V$ and $W$ together it must be $0=(0_{R_1},\,...,\,0_{R_n})\in R$. So $\dm(IA(R))=2$. We know that from corollary \ref{grth}, $\gr(IA(R))=3$. Namely, for $n\neq2$, $I_{R_1}-I_{R_2}-I_{R_3}-I_{R_1}$ is a cycle of length three. For $n=2$ we have $(a_1,\,0_{R_2})-(a_1,\,a_2)-(0_{R_1},\,a_2)-(a_1,\,0_{R_2})$ is also a cycle of length three.
\end{proof}

\begin{example}
  {Let $R=\mathbb{Z}_{4}\times \mathbb{Z}_{4}$. Then the set of zero divisors is $$Z(R/\sim)^*=\{[(0,\,1)]_{\sim},\,[(1,\,0)]_{\sim},\,[(0,\,2)]_{\sim},\,[(2,\,0)]_{\sim},\,[(2,\,2)]_{\sim},\,[(2,\,1)]_{\sim},\,[(1,\,2)]_{\sim}\}.$$ So the graph $IA(R)$ is as in figure~\ref{graph7}.
 \begin{figure}[H]%
    \centering%
   \[ \begin{tikzpicture}[scale=0.7,every node/.style={fill, circle,draw,scale=.5}]%
      \node(n1) at (-1,1){};
      \node(n2) at (1,2){};
      \node(n3) at (3,1){};
      \node(n4) at (-2,-1){};
      \node(n5) at (4,-1){};
      \node(n6) at (0,-3){};
      \node(n7) at (2,-3){};
      \draw[-] (n1)--(n2);
      \draw[-] (n1)--(n3);
      \draw[-] (n1)--(n4);
      \draw[-] (n1)--(n5);
      \draw[-] (n1)--(n6);
      \draw[-] (n1)--(n7);
      \draw[-] (n2)--(n3);
      \draw[-] (n2)--(n4);
      \draw[-] (n2)--(n5);
      \draw[-] (n2)--(n6);
      \draw[-] (n2)--(n7);
      \draw[-] (n3)--(n5);
      \draw[-] (n3)--(n4);
      \draw[-] (n3)--(n6);
      \draw[-] (n3)--(n7);
      \draw[-] (n4)--(n6);
      \draw[-] (n5)--(n7);
      \node[draw=none,rectangle, left=3mm,fill=none] (n1) at (-1,1) {$[(0,\,2)]_{\sim}$};
      \node[draw=none,rectangle, above=3mm,fill=none](n2) at (1,2)  {$[(2,\,2)]_{\sim}$};
      \node[draw=none,rectangle, right=3mm,fill=none](n3) at (3,1)  {$[(2,\,0)]_{\sim}$};
      \node[draw=none,rectangle, left=3mm,fill=none] (n4) at (-2,-1){$[(0,\,1)]_{\sim}$};
      \node[draw=none,rectangle, right=3mm,fill=none](n5) at (4,-1) {$[(1,\,0)]_{\sim}$};
      \node[draw=none,rectangle, below=3mm,fill=none](n6) at (0,-3) {$[(2,\,1)]_{\sim}$};
      \node[draw=none,rectangle, below=3mm,fill=none](n7) at (2,-3) {$[(1,\,2)]_{\sim}$};
      \end{tikzpicture}%
\]
      \caption{\label{graph7}$IA(R)$} {\footnotesize{$R=\mathbb{Z}_4\times \mathbb{Z}_4$}}%
     \end{figure}%
}
\end{example}

In the next theorem, we show that if $R$ is finite direct product of two rings such that one of them is not an integral domain, then $IA(R)$ is connected and not complete graph with $\dm(IA(R))\leq3$ and $\gr(IA(R))=3$

\begin{theorem}
  Let $R_1$ and $R_2$ be two commutative rings with non-zero identities such that $Z^*(R_1)\neq\phi$ or $Z^*(R_2)\neq\phi$ and $R=R_1\times R_2$. Then $IA(R)$ is connected and not complete graph with $\dm(IA(R))\leq3$ and $\gr(IA(R))=3$.
\end{theorem}
\begin{proof}
  Assume that $Z^*(R_1)\neq\phi$. Then there is $x\in Z^*(R_1)$ such that $rx=0$ for some $r\in Z^*(R_1)$. We claim that we have three different vertices in $Z(R/\sim)^*$ namely, $[(1_{R_1},0_{R_2})]_\sim$, $[(0_{R_1},1_{R_2})]_\sim$ and $[(x,0_{R_2})]_\sim$. It would follow by theorem~\ref{diam}, $IA(R)$ is connected and $\dm(IA(R))\leq3$ and by corollary~\ref{grth}, $\gr(IA(R))=3$. The vertices are distinct, since clearly $[(1_{R_1},0_{R_2})]_\sim\neq[(0_{R_1},1_{R_2})]_\sim$. On the other hand, $(r,0_{R_2})\in \ann(x,0_{R_2})$ and $(r,0_{R_2})\notin \ann(1_{R_1},0_{R_2})$. Then $[(x,0_{R_2})]_\sim\neq[(1_{R_1},0_{R_2})]_\sim$. Besides, $(0_{R_1},1_{R_2})\in \ann(x,0_{R_2})$ and $(0_{R_1},1_{R_2})\notin \ann(0_{R_1},1_{R_2})$. Thus $[(x,0_{R_2})]_\sim\neq[(0_{R_1},1_{R_2})]_\sim$. To show that $IA(R)$ is not a complete graph, it's easy to verify that $[(1_{R_1},0_{R_2})]_\sim$ and $[(0_{R_1},1_{R_2})]_\sim$ are not adjacent. Similarly, if $Z^*(R_2)\neq\phi$, then $IA(R)$ is connected and not complete graph with $\dm(IA(R))\leq3$ and $\gr(IA(R))=3$.
\end{proof}
\end{Section}
%
%

%

%

\begin{thebibliography}{100}
%
\bibitem{Module1}     A. Abbasi and A. Ramin, An extension of total graph over a module, \textit{Miskolc Math.\ Notes} \textbf{18(1)} (2017) 17-29. 
\bibitem{Non-comm}    S. Akbari and A. Mohammadian, {Zero-divisor graphs of non-commutative rings}, \textit{J. Algebra} \textbf{296} (2006), 462-479.
\bibitem{Total}       D.-F. Anderson and A. Badawi, {The total graph of a commutative ring}, \textit{{J. Algbra}} \textbf{320} (2008), 2706–2719.
\bibitem{GTotalGph}   D.-F. Anderson and A. Badawi, {The generalized total graph of a commutative ring}, \textit{{J. Algebra Appl.}} \textbf{12(5)} (2013), 1250212(1)-1250212(18).
\bibitem{LivingstonI} D.-F. Anderson and P.-S. Livingston, {The zero-divisor graph of a commutative ring}, {\textit{J. Algebra}} \textbf{217(2)} (1999), 434-447.
\bibitem{Atiyah}      M. F. Atiyah and I. G. Macdonald, \textit{Introducton to Commutative Algebra}, University of Oxford, Avalon Publishing, 1994.
\bibitem{Annih}       A. Badawi, {On the annihilator graph of a commutative ring}, {\textit{Commun.\ Algebra}} \textbf{42(1)} (2014), 108-121.
\bibitem{SurveyAB}    A. Badawi, On the total graph of a ring and its related graphs: a survey, in: M.~Fontana et al (Eds.), \textit{Commutative Algebra: Recent Advances in Commutative Rings, Integer-Valued Polynomials and Polynomial Functions}, Springer-Verlage, New York, 2014, 39–54.
\bibitem{SurveyAG}    A. Badawi, Recent results on the annihilator graph of a commutative ring: a survey, in: K.-S.~Prasad, K.-B.~Srinivas, P.~Harikrishnan and B.~Satyanarayana (Eds.), \textit{Nearrings, Nearfields and Related Topics}, World Scientific, 2017, 170-184. 
\bibitem{Beck}        I. Beck, {Coloring of commutative rings}, {\textit{J. Algebra}} \textbf{116(1)} (1988), 208--226.
\bibitem{ZeroModule}  M. Behboodi, {Zero-divisor graphs for modules over commutative rings}, {\textit{J. Commut.\ Algebra}} \textbf{4(2)} (2012), 175-197.
\bibitem{Module3}     F.-E.-K. Saraei, H.-H. Astaneh and R. Navidinia, {The total graph of a module with respect to multiplicative-prime subsets}, {\textit{Roma.\ J. Math.\ Comput.\ Sci.}} \textbf{4(2)} (2014), 151-166.
\bibitem{ZeroDivisor} M. Filipowicz and M. K\c{e}pczyk, {A note on zero-divisors of commutative rings}, {\textit{Arab.\ J. Math.}} \textbf{1} (2012), 191-194.
\bibitem{Elizebth}    E.-F. Lewis, \textit{The Congruence-Based Zero-Divior Graph}, PhD Disseretation, University of Tennessee, Knoxville, 2015.
\bibitem{Module4}     Sh. Ghalandarzadeh and P. Malakooti Rad, {Torsion graph of modules}, {\textit{Extracta math.}} \textbf{26(1)} (2011), 153-163.
\bibitem{Malakooti}   P. Malakooti Rad, S. Yassemi, Sh. Ghalandarzadeh and P. Safari, {Diameter and girth of torsion graph}, {\textit{Versita}} \textbf{22(3)} (2014), 127-136.
\bibitem{Mulay}       S.-B. Mulay, {Cycles and symmetries of zero-divisors}, {\textit{Commun.\ Algebra}} \textbf{30(7)} (2002), 3533-3558.
\bibitem{SurveyKN}    K. Nazzal, {Total graphs associated to a commutative ring}, {\textit{Palest.\ J. Math.}} \textbf{5(1)} (2016), 108-126.
\bibitem{Gen.ann}     S. Payrovi and S. Babaei, {The compressed annihilator graph of a commutative ring}, {\textit{Indian J. Pure Appl. Math.}} \textbf{49(1)} (2018), 177-186.
\bibitem{ideal-based} S.-P. Redmond, {An ideal-based zero-divisor graph of a commutative ring}, {\textit{Commun.\ Algebra}} \textbf{31(9)} (2003), 4425-4443.
\bibitem{compressed}  S. Spiroff and C. Wickham, {A zero divisor graph determined by equivalence classes of zero divisors}, \textit{{Commun.\ Algebra}} \textbf{39} (2011), 2338-2348.
\bibitem{Module5}     N.-K. Tohidi, F. Esmaeili-Khalil Saraei and S.-A. Jalili, {The generalized total graph of modules respect to proper submodules over commutative rings}, {\textit{J. Algebra and Related Topics}} \textbf{2(1)} (2014), 27-42.

\end{thebibliography}
\end{document}